\documentclass[11pt, twoside, notitlepage]{amsart}

\usepackage{amsmath,amsthm,amssymb}
\usepackage[utf8]{inputenc}
\usepackage{graphicx}
\usepackage{color}
\usepackage{hyperref}
\usepackage{nameref}
\usepackage{pgf,tikz}
\usepackage{pdfsync}

\newtheorem{theorem}{Theorem}[section]
\newtheorem{proposition}[theorem]{Proposition}

\newtheorem{remark}{Remark}
\newtheorem{defn}[theorem]{Definition}
\newtheorem{lemma}[theorem]{Lemma}

\newcommand\R{\mathbb{R}}
\newcommand\RN{\mathbb{R}^N}

\numberwithin{equation}{section}

\begin{document}

\title{
Fractional powers of the backward heat operator and Carleman type inequalities}

\author{Diana Stan}
\address[D. Stan]{University of Cantabria, Department of Mathematics, Statistics and Computation,
Avd. Los Castros 44, 39005 Santander, Spain}
\email{\href{mailto:diana.stan@unican.es}{\textnormal{diana.stan@unican.es}}}



\thanks{D. Stan is supported by the project  PID2020-114593GA-I00 financed by MCIN/AEI/10.13039/501100011033 and by the project "Ecuaciones en derivadas parciales motivadas por procesos de difusión y mecánica de fluidos: propiedades, asintóticas y homogeneización" (Ayuda financiada contrato Programa Gob. Cantabria -UC)}

\begin{abstract}
	In this paper we derive the  fractional power of the backward heat operator as a high dimensional limit of the fractional Laplacian. As applications, we  derive  Carleman type inequalities for  fractional powers of the backward heat operator.
\end{abstract}

\maketitle

\section{Introduction and main results}

Fractional powers of the heat operator  $(-\Delta+\partial_t)^s$  have garnered significant attention within the mathematical community in recent years. One basic reference is the paper of Metzler and Klafter \cite{Metzler} where $(-\Delta+\partial_t)^s$ is derived via a master equation.
The exploration of this concept has led to many works addressing its definition and fundamental properties \cite{StingaTorreaMastersEq},  study of partial differential equations involving it \cite{Audrito, Audrito-Terracini,Ferreira-DePablo} or obstacle problems \cite{Athanasopoulos}, among others.

Our research is motivated by the strong unique continuation property. This principle asserts that if
$u$ is a solution to  $(-\Delta+\partial_t)^s u=Vu$ and  $u$ vanishes of infinite order in a point $(x_0,t_0)$ then $u$ must be identically zero.
Various methodologies exist to establish unique continuation, among which Carleman-type inequalities and the monotonicity of the Almgren frequency formula. Both approaches work with suitable compactly supported or with sufficiently good decaying functions. However, in the case of fractional operators, establishing unique continuation poses a great challenge due to the  difficulty in localizing functions when subjected to multiplication by cut-off functions. Consequently, a significant number of published works on unique continuation employs the technique of the extension problem, wherein the solution to the nonlocal problem is derived as a limit of solutions to a local problem posed in a higher dimension.

Notable contributions to the unique continuation problem for the operator  $(-\Delta+\partial_t)^s$ include the work of Banerjee and  Garofalo \cite{BanerjeeGarofalo}. The authors study equations of the form $(-\Delta+\partial_t)^s u=Vu$ for $x\in \R^n $ and $t\in \R$, and $V(x,t)$ is a potential with certain regularity and decay properties.
They show that, if $u$ vanishes of infinite order backward in time at a point $(x_0,t_0)$  in the sense that
$$
  \text{esssup}_{B_r(x_0) \times (t_0-r^2, t_0]} |u(x,t)| = O(r^k), \qquad \text{for all } k>0,
$$
 then $u(\cdot, t) \equiv 0$ for all $t\le t_0$. In fact, in \cite{BanerjeeGarofalo2}, they show that $u\equiv 0$ in all $\R^n  \times \R.$
 These works predominantly rely on the Almgren monotonicity formula and use the technique of the extension problem, where $(-\Delta+\partial_t)^s u$ is derived as the limit to the boundary of a Neumann weighted derivative of the solution $U(x,y,t)$ to some extended problem in $\R^n \times [0,\infty) \times \R$ that takes boundary data $U(x,0,t)=u(x,t)$.

An alternative approach involves establishing Carleman-type inequalities, as it has been done for instance in the local case $s=1$ by Escauriaza \cite{EscauriazaDuke2000} and Fern\'andez  \cite{FernandezCPDEs2003}.

A different  but significantly motivating topic for our work is the  paper \cite{DaveyARMA2018} by Davey, in which the author proves that the backward heat operator $(\Delta_x+\partial_t)[u(x,t)]$  for $x\in \R^d$ can be obtained as a high dimensional limit as $N\to \infty$ of the Laplacian $\Delta_{x,y} \Big[u\big(x,\frac{|y|^2}{2N}\big) \Big]$ for $x\in \R^d $ and $y\in \R^N$.  This relation has many useful applications as the author shows in \cite{DaveyARMA2018}, and afterwards in \cite{DaveyGarcia}. Related to the motivation of our current paper, in \cite{DaveyARMA2018}, the author shows that starting from the Carleman type estimate for $\Delta_{x,y}$ proved by Amrein, Berthier and Georgescu in \cite{Amrein}, one can derive a Carleman type estimate for $\Delta_x+\partial_t$ which is the same as the one proved by Escauriaza in \cite{EscauriazaDuke2000}.

In this paper we work with fractional powers of the Laplace and backward heat operators by defining them via the subordination formula. For suitable functions $u:\R^d\times [0,\infty) \to \R$, we prove  in Proposition \ref{prop:transf} that,
if $G_N:\R^d \times [0,\infty) \to \R$ is  the notation for
  $$G_N\big(x, \frac{|y|^2}{2N}  \big):= (-\Delta_{x,y})^s \left[ u\big(x,\frac{|y|^2}{2N}\big)  \right] ,$$ then
  $$G_{N} (x,t) \rightarrow  ( -\Delta_x-\partial_t)^{s} [u(x,t)] \quad \text{as }N\to \infty. $$
Then, by using some weighted Hardy-Rellich type inequalities  for $(-\Delta_{x,y})^s$ proved by Eilertsen \cite{EilertsenJFA2001}  and by de Nitti and Djitte \cite{DeNitti-FracIneq-2024}, both with explicit constants, we are able to deduce the following Carleman type inequalities for  the backward in time  operator $( -\Delta_x-\partial_t)^s$. Our choice to work with the backward heat operator rather than the heat operator is motivated by the strong unique continuation results proved via Carleman-type inequalities available in the local case $s=1$ \cite{EscauriazaDuke2000,FernandezCPDEs2003}.


For functions $u : \R^d \times [0,\infty)$ we define the even extension with respect to the last variable as
\begin{equation}\label{even-extension}
 \tilde{u}(x,t)=  \begin{cases}
                     u(x,t),& \quad  \mbox{if }  \quad x\in \R^d, \,  t\ge 0,  \\
                     u(x,-t), &  \quad  \mbox{if } \quad   x\in \R^d,\,  t< 0.
                   \end{cases}
\end{equation}
We will denote by $\mathcal{S}(\R^{d+1})$ the Schwartz space of smooth rapidly decaying functions.

\begin{theorem}\label{main:thm} Let $s  \in (0,1)$  and $\eta >0$ such that $ 2s+\eta\not\in \mathbb{N}$.
Let $u : \R^d \times [0,\infty)$ such that $\tilde{u} \in \mathcal{S}(\R^{d+1})$ ,  where $\tilde{u}$  is defined in $\eqref{even-extension}.$
Then
\begin{align}\label{Carleman}
	&\int_{\R^{d}}\int_0^\infty |u(x,t )|^2 \cdot e ^{-\frac{|x|^2}{4t}}  \cdot   t^{\frac{  2\eta-d -2 }{2}} \, dt\,  dx \\ \nonumber
	&\le  C\int_{\R^{d}}\int_0^\infty  \left| ( -\Delta_x-\partial_t)^{s} [u(x,t)]  \right|^2  \cdot e ^{-\frac{|x|^2}{4t}}  \cdot t^{ 2s + \frac{2\eta -d-2}{2}} \, dt \, dx.
\end{align}
Here, $C$ is a positive constant given by $$
C = \max_{j \in \{ 0,1,\dots, j_1+1 \}} \left(  \frac{ \Gamma\big( \frac{\eta+j}{2} \big)   }
{ \Gamma\big( \frac{2s+\eta+j}{2} \big)} \right)^2,$$
where $j_1$ to be the smallest non-negative integer satisfying
$$0\le j_1 + s + \eta  \quad  \text{and}  \quad    \eta(2s+\eta) \le   j_1^2  . $$
\end{theorem}
Negative values of $\eta$ can be also considered, as long as the function $(x,y)\mapsto u\big(x,\frac{|y|^2}{N}\big)$ satisfies the hypothesis of Lemma \ref{Lemma:Eilertsen} for all $N\in \mathbb{N}$.

A generalization to $L^p$ norms is presented in the next theorem, although we establish it only for compactly supported functions.

\begin{theorem}\label{Thm:CarlemanLp}
Let $d\ge 1$, $s\in (0,1)$,  $1<p<+\infty$, $\eta  \in \R$ such that $d-2\eta>0$. Then
    \begin{align}\label{Carleman-Lp}
&  \left(\frac{\tilde{b}_{d,\eta,s}}{p}\right)^p \int_{\R^{d}}\int_0^\infty |u(x,t )|^p   \cdot e ^{-\frac{|x|^2}{4t}}  \,   t^{\frac{  2\eta-d -2 }{2}}  dt \, dx \\
\nonumber &\le   \int_{\R^{d}}\int_0^\infty  \left| ( -\Delta_x-\partial_t)^{s} [u(x,t)]  \right|^p  \cdot e ^{-\frac{|x|^2}{4t}} \,  t^{ sp + \frac{2\eta -d-2}{2}} dt\, dx,
\end{align}
for all $u : \R^d \times[0,\infty) \to \R$ such that $ \tilde{u} \in S(\R^{d+1})   \cap C^{1,1}_c (\R^{d+1})$, where $\tilde{u}$  is defined in $\eqref{even-extension}.$ The constant is given by $\tilde{b}_{d,\eta,s}=\frac{
  \Gamma\left(  \frac{ - 2\eta-2s}{2} \right)
  }
  {
   \Gamma \left(\frac{2\eta-d}{2} \right)
   } $.
\end{theorem}

\begin{remark} \begin{enumerate}
              \item[(i)] If $s=1$ and $\frac{  2\eta-d -2 }{2} <0$ in \eqref{Carleman}, we recover the result  proved by Escauriaza in \cite[Thm.1]{EscauriazaDuke2000} for the backward heat operator $\Delta_x+\partial_t$.
\item[(ii)] To our knowledge, there are no references that address Carleman inequalities for the fractional operator $(-\Delta-\partial_t)^s$ and we believe that estimates \eqref{Carleman}  and \eqref{Carleman-Lp} are new.
                   \end{enumerate}
\end{remark}

\section{Notations and preliminaries}

 \subsection{The functional space} We denote by $\mathcal{S}(\mathbb{R}^n)$  the Schwartz class of real valued smooth rapidly decaying functions.  Let $\mathcal{L}(\mathbb{R}^n)$ be the Lizorkin space (see \cite{EilertsenJFA2001}, \cite[Section 4]{Troyanov}), that is the closed subspace of $\mathcal{S}(\mathbb{R}^n)$ of functions that have all the $k$ moments equal to zero:
$$\mathcal{L}=\left\{ f \in \mathcal{S}(\mathbb{R}^n):   \partial^{k} \widehat{f}(0)=0, \, \forall k \in \mathbb{N}^n  \right\} = \left\{ f \in \mathcal{S}:  \int_{\mathbb{R}^n } f(x) |x|^k dx= 0, \, \forall k \in \mathbb{N}^n  \right\}. $$
It is known that the closure of $\mathcal{L}$ in the $L ^\infty$ norm is the $C_0^{\infty}(\mathbb{R}^n)$, the space of smooth functions that vanish at infinity (see \cite[Thm.3.2]{Neumayer-Unser}).
$\mathcal{L}$ is a good space for defining the fractional Laplacian, since in this space the operator becomes invertible.
 In Fourier space, $\mathcal{L}$ can be seen as:
 $$\widehat{\mathcal{L}} (\mathbb{R}^n)= \{ \hat{f}: f\in \mathcal{L}(\mathbb{R}^n) \} = \{  \psi \in \mathcal{S}: \partial^{k} \psi(0)=0, \, \forall k \in \mathbb{N}^n\},$$
 where $\widehat{f}$ is the Fourier transform of $u$ and $k$ is a multi-index.   When the expression is too large, we denote the Fourier transform by $\mathcal{F}$.  The Lizorkin space is need only in the hypothesis of Lemma \ref{Lemma:Eilertsen} in the case $\eta<0$, otherwise we work only in the Schwartz class.

\subsection{The fractional Laplacian}  
The fractional Laplacian $(-\Delta)^s f$ can be represented by many equivalent definitions, among which we recall three of them \cite{Bucur,Kwasnicki,Garofalo-Thoughts}:
\begin{enumerate}
\item[(i)]  Via the Fourier transform:  for $f \in \mathcal{S}(\mathbb{R}^n)$ and $s \in \mathbb{R}$, $(-\Delta)^s f $ is defined such that
    $$  \mathcal{F}( (-\Delta)^s f ) (\xi)= |\xi|^{2s} \hat{f}(\xi).$$
    Then $(-\Delta)^s u : \mathcal{L} \to \mathcal{L}$ is a bijection.
    Moreover, if $s>-n/2$, then  $(-\Delta)^s f (x)= O(|x|^{-n-2s})$ when $|x| \to +\infty$, for any $f \in \mathcal{S}(\mathbb{R}^n)$.

\item[(ii)] Via the subordination formula, also known as Bochner's integral: for $s \in (0,1) $
\begin{equation}\label{FracLapSub}
(-\Delta)^s f(x)=\frac{1}{|\Gamma(-s)|} \int_0^ \infty \left( e^ {-\tau\Delta} f(x) - f(x) \right) \frac{d\tau}{\tau^ {1+s}}.
\end{equation}
The inverse fractional Laplacian is the Riesz potential  that, for $0<s<n$,  is given by:
\begin{equation}\label{InverseFracLapSub}
(-\Delta)^{-s} f(x)=\frac{1}{\Gamma(s)} \int_0^ \infty  e^ {-\tau\Delta} f(x)  \frac{d\tau}{\tau^ {1-s}},
\end{equation}
where $e^ {-\tau\Delta} f(x)$ is the heat semigroup with initial data $f$ and time variable $\tau$, that is  $U(x, \tau)= e^ {-\tau\Delta} f(x)$ is the solution to the heat equation with data $f(x)$:
$$
\frac{d}{d\tau} U(x,\tau)=\Delta U(x,\tau), \quad  U(x,0)=f(x),\quad x\in \R^n.
$$
\item[(iii)] Via the Levy-Khintchine integral representation:
\begin{equation}\label{def:fractLap}
(-\Delta)^s f(x)=  \frac{4^s \Gamma(n/2+s)}{\pi^{n/2} |\Gamma(-s)|}\,  \text{P.V.} \int_{\R^n} \frac{f(x)-f(y)}{|x-y|^{n+2s}}dy.
\end{equation}
\end{enumerate}
It is known that the fractional Laplacian of radial functions is also radial. In the following lemma, we show that, if the function is radial with respect to only some of the coordinates, then its fractional Laplacian has the same property.

In the following we will use the notation $\Delta_{x,y} =\Delta_x + \Delta_y =\sum_{i=1}^d \partial_{x_i}^2  + \sum_{i=1}^N \partial_{y_i}^2  $ applied to functions $f: \R^{d+N}\to \R$.
\begin{lemma} \label{Lemma:radial}
Let $u:\mathbb{R}^{d} \times [0,\infty) \to  \mathbb{R}$. Let $f:  \mathbb{R}^{d+N} \to  \mathbb{R}$ be defined as
$$f(x,y)=u(x,|y|).$$
Then $(-\Delta_{x,y})^s [f(x,y)]$ is radial with respect to the $y$ coordinates.
\end{lemma}
\begin{proof}
The proof can be done using definition \eqref{def:fractLap}.
Let $x\in\mathbb{R}^{d}$ and $  y\in \mathbb{R}^{N}$. Let $R \in \mathbb{R}^{N\times N}$ be a rotation matrix in $\mathbb{R}^{N}$. Then $|Rz|=|z|$ for all $z\in \mathbb{R}^{N}$.
We have:
\begin{align*}
(-\Delta_{x,y} )^s [f](x,Ry)&= C \,  \text{P.V.} \int_{\R^d}\int_{\R^N} \frac{u(w,|z|)-u(x,|Ry|)}{ |(w,z)-(x,Ry)|^{N+d+2s} }dy \, dx.
\end{align*}
We perform the change of coordinates: $z=R\, \overline{z}$. Thus $|z|=|\overline{z}|$,   $dz= d\overline{z}$ and
$$ |(w,R\overline{z})-(x,Ry)|^2 = |w-x|^2 +  |R\cdot (\overline{z}-  y)|^2 = |w-x|^2 +  |\overline{z}-  y|^2 = |(w,\overline{z})-(x,y)|^2 ,$$
and then we obtain
\begin{align*}
(-\Delta_{x,y} )^s [f](x,Ry)&=  C   \,  \text{P.V.} \int_{\R^d}\int_{\R^N} \frac{u(w,|\overline{z}|)-u(x,|y|)}{|(w,\overline{z})-(x,y)|^{N+d+2s}}dy \, dx \\
&=(-\Delta_{x,y} )^s [f](x,y).
\end{align*}
Additionally, in formula \eqref{radial:U-N}, we will see an alternative proof using the subordination formula.
\end{proof}

\subsection{Weighted estimates for the fractional Laplacian}
We state two fractional Hardy-Rellich type inequalities  for the fractional Laplacian. The first one was obtained by Eilertsen \cite{EilertsenJFA2001}. The class of functions for which this inequality applies is explained in \cite[p.362]{EilertsenJFA2001} and we recall it in the hypothesis of Lemma \ref{Lemma:Eilertsen}.


\begin{lemma}\label{Lemma:Eilertsen}\cite[Thm. 16]{EilertsenJFA2001}
Let $s, \eta \in \mathbb{R}$ such that $s \ge 0$  and  $-\eta,\, 2s+\eta-n\not\in \mathbb{N}$. Let $f \in \mathcal{S}(\R^n)$ and, moreover, if $\eta \le 0$ then $f$ should also be in $\widehat{\mathcal{L}}$. Then
\begin{equation}\label{ineq:Eilertsen}
	\int_{\R^{n}} f^2(x) |x|^{2\eta - n} dx  \le C \int_{\R^{n}} | (-\Delta)^{s}  f(x)|^2  |x|^{ 4s +2\eta - n} dx.
\end{equation}
Let $j_0$ be the smallest non-negative integer satisfying $(n - 2\eta )( n-4s-2\eta ) \le (n + 2j_0)^2$  when $n\ge 2$ and $j_0=0$ if $n=1$.  Then
	the best constant $C$ is given by
\begin{equation}\label{CNd}
 C_{n}= 2^{-4s} \max_{j=j_0, j_0+1} \left(  \frac{ \Gamma\big( \frac{\eta+j}{2} \big)   \Gamma\big( \frac{n-2s-\eta+j}{2} \big)   }
{ \Gamma\big( \frac{2s+\eta+j}{2} \big)   \Gamma\big( \frac{n-\eta+j}{2} \big)   } \right)^2.
\end{equation}
\end{lemma}

The inequality above  is related to the Stein–Weiss inequality \cite[Thm.B*]{Stein-Weiss}, but in their paper the constant is not computed explicitly, which for us is important.

The following inequality  has been proved by De Nitti and Djitte  and can be seen as a generalization of \eqref{ineq:Eilertsen} to $L^p$ norms and $C^{1,1}_c$ functions.

\begin{lemma}\label{Lemma:deNitti}(\cite[Remark 2.5, item 2]{DeNitti-FracIneq-2024})
Let $s\in (0,1)$,   $\theta>-2s$, $1<p<+\infty$. Then
 \begin{equation*}
   \left[   \frac{b_{n,s,\theta}}{p}\right]^p \int_{\R^n} \frac{|f|^p}{|x|^{\theta+2s}} dx \le \int_{\R^n}\frac{| (-\Delta)^s f |^p}{|x|^{\theta+2s-2sp}} dx
 \end{equation*}
  for all $f \in C^{1,1}_c(\R^N).$  The constant is given by
\begin{equation*}
  b_{N,s,\theta}=2^{2s}
  \frac{
  \Gamma\left(  \frac{n-\theta}{2} \right) \cdot  \Gamma\left( \frac{2s+\theta}{2}\right)
  }
  {
   \Gamma \left(\frac{n-\theta-2s}{2} \right) \cdot  \Gamma\left(  \frac{\theta}{2}\right)
   }.
\end{equation*}
\end{lemma}



\subsection{The fractional power of the backward heat operator.}  For $u:\R^d\times \R  \to \R$,
$(-\Delta-\partial_t)^s u $ can be represented via several equivalent definitions (see \cite{StingaTorreaMastersEq,BanerjeeGarofalo}):
\begin{enumerate}
\item[(a)] Via the Fourier transform:
\begin{equation}
	\mathcal{F}_{x,t}((-\Delta-\partial_t)^s u(t,x)) (\tau,\xi)=(|\xi|^2-i \tau)^s \widehat{h}(\tau,x), \quad \xi \in \R^d, \tau\in \R.
\end{equation}
\item[(b)] The subordination formula, also known Bochner's subordination formula or Balakrishnan formula (see [p.260, formula (5)]\cite{Yosida},  [p.70 formula (6.9)]\cite{Pazy}):
\begin{equation}\label{def:FracHeatOperator}
(-\Delta-\partial_t)^s u(t,x) =\frac{1}{\Gamma(-s)}\int_0^\infty \left(e^{\tau (-\Delta-\partial_t) } u(t,x)- u(t,x)\right) \frac{d\tau}{\tau^{1+s}}, \quad \text{for } \, s\in (0,1),
\end{equation}
\begin{equation}\label{def:InverseFracHeatOperator}
(-\Delta-\partial_t)^{-s} u(t,x) =\frac{1}{\Gamma(s)}\int_0^\infty e^{\tau (-\Delta-\partial_t) } u(t,x) \frac{d\tau}{\tau^{1-s}}\quad \text{for } \, s\in (0,d).
\end{equation}
Here $U=e^{\tau (-\Delta-\partial_t) } u(t,x)$ is the solution to the problem
\begin{equation}\label{eq2:U-RN}
  \left\{ \begin{array}{ll}
  \frac{d}{d\tau}U(x,t,\tau)=\Delta_{x} U(x,t,\tau) +  U_t(x,t,\tau) &\text{for }\tau>0,\, t \in  \R \text{ and } x \in \R^d,  \\[2mm]
  U(x,t,0)=u(x,t) &\text{for } x \in \R^d, \, t \in  \R .
    \end{array}
    \right.
\end{equation}
\end{enumerate}

When $u(t,x)$ is not dependent of $t$ then we recover $(-\Delta)^s u(x)$. When $u$ does not depend on $x$ we recover the fractional derivative in time  called Marchaud.

\begin{defn}\label{defn:FractBackwardHeat}
For functions $u:\R^d\times [0,\infty)$ we consider $(-\Delta-\partial_t)^s u$ defined by \eqref{def:FracHeatOperator}, respectively \eqref{def:InverseFracHeatOperator}, where $V(x,t,\tau):= e^{\tau (-\Delta-\partial_t) } u(t,x)$ is the solution to the initial boundary value  problem
\begin{equation}\label{eq2:V-bis}
  \left\{ \begin{array}{ll}
  \frac{d}{d\tau}V(x,t,\tau)=\Delta_{x} V(x,t,\tau) +  V_t(x,t,\tau) &\text{for }\tau>0,\, t >0 \text{ and } x \in \R^d,  \\[2mm]
  V(x,0,\tau)= \frac{     1    }{ (4 \pi \tau)^{d/2}  }  \int_{\R^d}  e^{ - \frac{|x-\overline{x}|^2 }{4 \tau} }  u(\overline{x}, \tau ) \, d \overline{x}  &\text{for } x \in \R^d, \, \tau>0, \\[2mm]
  V(x,t,0)=u(x,t) &\text{for } x \in \R^d, \, t\ge0.
    \end{array}
    \right.
\end{equation}
\end{defn}
In Lemma \ref{Lemma:V} we prove the solution of \eqref{eq2:V-bis} is unique (as a bounded solution) and coincides with the solution of \eqref{eq2:U-RN} on $\R^d \times [0,\infty) \times [0,\infty)$, thus definition \eqref{defn:FractBackwardHeat} coincides with \eqref{def:FracHeatOperator} for functions $u$ that are globally defined.

\section{Transformation}

\begin{proposition}\label{prop:transf} Let $s \in (-d,1)$.  Let $u : \R^d \times [0,\infty)$ such that $\tilde{u} \in \mathcal{S}(\R^{d+1})$,  where $\tilde{u}$  is defined in $\eqref{even-extension}.$
   \begin{itemize}
     \item[(i)]Then   $(-\Delta_{x,y})^s \left[ u\big(x,\frac{|y|^2}{2N}\big)  \right]$ is a function only of $\big(x,\frac{|y|^2}{2N}\big)$, for every $N\in  \mathbb{N}$, $N\ge 1$. \\[1mm]
  \item[(ii)]  Let   $G_N:\R^d \times [0,\infty) \times (0,\infty) \to \R$  such that
  $$G_N\big(x, \frac{|y|^2}{2N}  \big)= (-\Delta_{x,y})^s \left[ u\big(x,\frac{|y|^2}{2N}\big)  \right] .$$
  \noindent Then
  $$G_{N} (x,t) \rightarrow  ( -\Delta_x-\partial_t)^{s} [u(x,t)] \quad \text{as }N\to \infty. $$
   \end{itemize}
  \end{proposition}
\begin{proof}
We define the function  $f:\R^d \times \RN \to \R$  by
$$f(x,y)=u\big(x,\frac{|y|^2}{2N}\big) ,$$
and, thus,  $f$ is radial with respect to the $y$ coordinate.
The idea of this definition, as well as the subsequent limit process, is inspired by the paper of Davey \cite{DaveyARMA2018} and by Tao's blog \cite{Tao}. We present the proof for $s\in (0,1)$; however for $s\in (-d,0)$ similar arguments apply using the corresponding definitions \eqref{InverseFracLapSub} and \eqref{def:InverseFracHeatOperator} for $(-\Delta_{x,y})^{-s}, ( -\Delta_x-\partial_t)^{-s}$.  The case $s=0$ is straightforward and is related to \cite{DaveyARMA2018}.

\noindent\emph{Proof of (i). The fractional Laplacian. }
We observe first that $(-\Delta_{x,y})^s f(x,y)$ is also radial with respect to the $|y|$ coordinate as we proved in Lemma \ref{Lemma:radial}, thus $(-\Delta_{x,y})^s f(x,y)=g(x,|y|)$ for all $y\in \R^N$, where $g$ is some function $\R^d \times [0,\infty)\to \R$ . Since $\varphi:[0,\infty) \to [0,\infty)$,  $\varphi(r)=r^2/(2N)$ is a bijection then $g(x,|y|)=
g(x,\varphi^{-1}(\varphi(y))$, and thus,
by defining
$$G_N:\R^d \times [0,\infty)\to \R,\quad G_N(x,t)=g(x,\sqrt{2N t}),$$
we have
\begin{equation}\label{FractLap:f}
(-\Delta_{x,y})^s f(x,y) =   G_N(x,  \frac{|y|^2}{2N}).
\end{equation}
Later in this proof, we will alternatively deduce this dependence  using the subordination definition.

Our goal is to prove that $G_N(x, t) \rightarrow (-\partial_t-\Delta_x)^s [u(x,t)]$ as $N\to \infty$ pointwise for every $(x,t) \in \R^d\times [0,\infty).$  We will approach this by using the semigroup definition of fractional powers of operators.
  By using definition \eqref{FracLapSub},
$$
(-\Delta_{x,y})^s f(x,y)=\frac{1}{\Gamma(-s)} \int_0^ \infty ( U_N(x,y,\tau) - f(x,y)) \frac{d\tau}{\tau^ {1+s}},
$$
where $U_N(x,y,\tau)$ solves the heat equation
\begin{equation} \label{eq:U}
  \left\{ \begin{array}{ll}
  \frac{d}{d\tau} U_N(x,y,\tau)=\Delta_{x,y} U(x,y,\tau) &\text{for }\tau>0 \text{ and } x \in \R^d, \, y \in \R^N, \\[2mm]
  U_N(x,y,0)  =f(x,y)=u(x, \frac{|y|^2}{2N} ) &\text{for } x \in \R^d, \, y \in \RN.
    \end{array}
    \right.
\end{equation}
 Recall that  $u$ is a bounded function since $u\in \mathcal{S}$, thus there exists $M>0$ such that $|u(x,t)| \le M,$ for all $  x\in \R^d, t\ge0.$  The Maximum Principle for the Heat Equation  (see \cite[p.57]{Evans}) 
gives us that
$$ \inf_{\R^d \times  \R^N} f(x,y)  \le \inf_{\R^d \times  \R^N}  U_N(x,y,t) \le  U_N(x,y,t)\le \sup_{\R^d \times  \R^N} U_N(x,y,t) \le \sup_{\R^d \times  \R^N} f(x,y).$$
Since $f(x,y)=u(x, \frac{|y|^2}{2N} )$ and $u$ is bounded, then we obtain that  $|U_N|$ is bounded on $\R^d \times \R^N$ with an upper bound independent of $N$.  \normalcolor

Since the data depends only on its values for $x$ and $\frac{|y|^2}{2N}$, then, it can be proved that also $U_N$ can be written as function of $(x, \frac{|y|^2}{2N},\tau )$.  This can be easy justified by the fact that  $U_N$ is radial with respect to $y$ and noticing that $ r \mapsto \frac{r^2}{2N}$ is a bijection from $[0, + \infty) \to [0, + \infty)$.  We can also give an explicit representation formula for this fact.
Indeed, using the integral representation  \eqref{def:fractLap}, we have:
\begin{align*}
  U_N(x,y,\tau) & = \frac{1}{(4 \pi \tau)^{(N+d)/2}}  \int_{\R^d} \int_{\R^N} e^{ - \frac{|x-\overline{x}|^2 + |y-\overline{y}|^2}{4 \tau} } u\big(\overline{x}, \frac{|\overline{y}|^2}{2N} \big) \, d \overline{y}\, d \overline{x}.
  \end{align*}
  We pass to polar coordinates $\overline{y}=r \sigma$ with $r\in [0,\infty)$, $\sigma \in S^{N-1}$ the unit sphere in $\mathbb{R}^N$ with area recalled in \eqref{area:sphere}. Thus
    \begin{align*}
  U_N(x,y,\tau)
   &= \frac{1}{(4 \pi \tau)^{(N+d)/2}}  e^{ - \frac{ |y|^2 }{4 \tau} }   \int_{\R^d}  \int_0^\infty
   e^{ - \frac{|x-\overline{x}|^2 }{4 \tau} } e^{ - \frac{  r^2}{4 \tau} }  r^{N-1 }u\big(\overline{x}, \frac{r^2}{2N} \big)
   \int_{S^{N-1}} e^{  \frac{ r  y \cdot \sigma}{2 \tau} } \, d\sigma  \, d r\, d \overline{x}.
\end{align*}
In order to compute the integral over the sphere for fixed $y$, a rotation argument helps to pass from $y\cdot \sigma$ to $|y| \vec{e}_1 \cdot \sigma  = |y| \sigma_1$, where $\sigma_1$ is the first component of $\sigma$ as a vector in $\R^N$. This is explained later on in the Appendix, see \eqref{Integral:Rotation}.
Thus
\begin{align*}
  \int_{S^{N-1}}   e^{  \frac{ r  y \cdot \sigma}{2 \tau} }  d\sigma&= \int_{S^{N-1}}   e^{  \frac{ r  |y| \sigma_1}{2 \tau} }  d\sigma =  |S^{N-2}| \int_{-1}^1   e^{  \frac{ r  |y| t }{2 \tau} } (1-t^2)^{\frac{N-2}{2}} \, d t  \\[2mm]
  &= \frac{\Gamma(\frac{N}{2} )}{\Gamma(\frac{N-1}{2})}  \cdot \frac{ 2^N \pi^{\frac{N}{2} }  \tau^{\frac{N-1}{2}}  }{(r  |y|   ) ^{\frac{N-1}{2}}} \cdot  I_{\frac{N-1}{2}} \Big(\frac{ r  |y|  }{2 \tau} \Big),
\end{align*}
where $I$ is the modified Bessel function of first kind that we recall in \eqref{BesselI}.
Then
\begin{align}\label{radial:U-N}
  U_N(x,y,\tau) & = \frac{1}{(4 \pi \tau)^{(N+d)/2}}
  \frac{\Gamma(\frac{N}{2} )}{\Gamma(\frac{N-1}{2})}   2^N \pi^{\frac{N}{2} }
    \tau^{\frac{N-1}{2}}  e^{ - \frac{ |y|^2 }{4 \tau} }  \\ \nonumber
  &\quad \cdot \int_{\R^d}  \int_0^\infty
   e^{ - \frac{|x-\overline{x}|^2 }{4 \tau} } e^{ - \frac{  r^2}{4 \tau} }  r^{N-1 }u(\overline{x}, \frac{r^2}{2N} ) \frac{ 1 }{(r  |y|   ) ^{\frac{N-1}{2}}}  I_{\frac{N-1}{2}}\Big( \frac{r  |y|  }{2 \tau} \Big) \,d r d \, \overline{x}.
\end{align}
If we denote $\frac{|y|^2 }{2N}=t$  and change variables $\frac{r^2}{2N}=p$  then $r dr=N dp$:
\begin{align*}
 & U_N(x,y,\tau)
   =\frac{1}{ 2^d     \pi^{\frac{d}{2} } \tau^{(d+1)/2}}    \frac{\Gamma(\frac{N}{2} )}{\Gamma(\frac{N-1}{2})}   N^{\frac{1}{2}} \cdot  \\
   & \qquad \int_{\R^d}e^{ - \frac{|x-\overline{x}|^2 }{4 \tau} }  \int_0^\infty u(\overline{x}, p )
      e^{ - \frac{  N }{2 \tau} (t+p)}    p^{\frac{N-2}{2} }  (pt)   ^{ - \frac{N-1}{4}}  I_{\frac{N-1}{2}}\Big(\frac{ N  \sqrt{pt}  }{ \tau} \Big) \, d p\, d \overline{x}.
\end{align*}
Thus $$ U_N(x,y,\tau) = V_N(x,\frac{|y|^2 }{2N},\tau), \quad \forall x  \in \R^d, \, y\in \R^N,\, \tau>0,$$
where $V_N:\R^d \times [0,\infty) \times [0,\infty)\to \R$ is defined by
\begin{multline}
  \label{def:VN}
 V_N(x,t,\tau) = \frac{1}{ 2^d     \pi^{\frac{d}{2} } \tau^{(d+1)/2}}    \frac{\Gamma(\frac{N}{2} )}{\Gamma(\frac{N-1}{2})}   N^{\frac{1}{2}} \\
 \cdot \int_{\R^d}e^{ - \frac{|x-\overline{x}|^2 }{4 \tau} }  \int_0^\infty u(\overline{x}, p )
      e^{ - \frac{  N }{2 \tau} (t+p)}    p^{\frac{N-2}{2} }  (pt)   ^{ - \frac{N-1}{4}}  I_{\frac{N-1}{2}}\Big(\frac{ N  \sqrt{pt}  }{ \tau} \Big)\, d p \,d \overline{x}.
\end{multline}
%
Equivalently,
$$
V_N(x,t,\tau)=U(x,\sqrt{2N  t}\, z, \tau)$$ for every $z\in \RN$ with modulus $1$.


\medskip

\noindent\emph{(ii). Passing to the limit.}  We use  the definition \eqref{def:VN} and the heat equation \eqref{eq:U} solved by $U_N$.
Let $\varphi:[0,\infty) \to [0,\infty)$,  $\varphi(r)=r^2/(2N)$, and thus
$$U_N(x,y,\tau) =V_N(x,\varphi(r),\tau), \quad  r=|y|.$$
By passing $\Delta_y$ to radial coordinates, it follows that $V_N(x,\varphi(r),\tau)$ solves the equation
\begin{align*}
\frac{d}{d\tau}V_N(x,\varphi(r),\tau)&=\Delta_{x} V_N(x,\varphi(r),\tau) +  (V_N)_t(x,\varphi(r),\tau) \cdot  \left(\varphi''(r) + \frac{N-1}{r} \varphi'(r)\right) \\
&\qquad+ (V_N)_{tt} (\varphi(r),\tau) \cdot (\varphi'(r))^2 ,
\end{align*}
where $(V_N)_t$ is the derivative of $V_N(x,t,\tau)$ with respect to the $t$ variable.
Since  $\varphi''(r) + \frac{N-1}{r} \varphi'(r)=1$, then  $V_N$ is a solution to the equation
\begin{align*}
\frac{d}{d\tau}V\big(x,\frac{r^2}{2N},\tau\big)&=\Delta_{x} V\big(x,\frac{r^2}{2N},\tau\big) +  (V_N)_t\big(x,\frac{r^2}{2N},\tau\big) + V_{tt} \big(x,\frac{r^2}{2N},\tau\big) \cdot \left(\frac{r}{N}\right)^2 .
\end{align*}
Let $r^2/(2N)=t$. Then,
\begin{equation}\label{eq:VN}
\frac{d}{d\tau}V_N(x,t,\tau)=\Delta_{x} V(x,t,\tau) +  (V_N)_t(x,t,\tau) + (V_N)_{tt} (x,t,\tau) \cdot \frac{t}{N},
\end{equation}
for $x\in \R^d$, $t\ge 0$, $\tau>0.$


Our purpose is to prove that $ V_N(x,t,\tau) $ is pointwise convergent as $N\to \infty$  for every $x\in \R^d$, $t\ge 0$, $\tau>0$.
We will use the following  property from calculus: a bounded sequence converges to $l$ if and only if any convergent subsequence converges to $l$.

We fix $(x,t,\tau) \in \R^d \times [0,\infty) \times (0,\infty).$  Let $(V_{N_k}(x,t,\tau) )_k$ be a convergent subsequence of $(V_N(x,t,\tau) )_N$ and we denote by $V(x,t,\tau)=\lim_{k\to \infty}V_{N_k} (x,t,\tau)$.
 We pass to the limit as $N_k\to \infty$ in the PDE solved by $V_{N_k}$ in order to obtain a PDE for $V$.

The lateral boundary value for $V_N$ is
\begin{align*}
  V_N(x,0,\tau) & =U_N(x,\vec{0},\tau)   =
  \frac{1}{(4 \pi \tau)^{(N+d)/2}}  \int_{\R^d} \int_{\R^N} e^{ - \frac{|x-\overline{x}|^2 + |\overline{y}|^2}{4 \tau} } u(\overline{x}, \frac{|\overline{y}|^2}{2N} ) d \overline{y} d \overline{x} \\
  &=
   \frac{ |S^{N-1}|}{(4 \pi \tau)^{(N+d)/2}}    \int_{\R^d}  \int_0^\infty
   e^{ - \frac{|x-\overline{x}|^2 }{4 \tau} } e^{ - \frac{  r^2}{4 \tau} }  r^{N-1 }u(\overline{x}, \frac{r^2}{2N} )
    d r d \overline{x} \\
   &= \frac{ |S^{N-1}|}{(4 \pi \tau)^{(N+d)/2}}      \int_{\R^d}  \int_0^\infty
   e^{ - \frac{|x-\overline{x}|^2 }{4 \tau} } e^{ - \frac{ N  t}{2 \tau} }  (2Nt)^{\frac{N-2}{2}  }  u(\overline{x}, t )
 N \, dt \, d \overline{x} \\
  &=   \frac{     1    }{ (4 \pi \tau)^{d/2}  } \cdot \frac{1}{ \Gamma(N/2)}
  \left(  \frac{N}{2\tau} \right)^{N/2} \cdot
  \int_{\R^d}  e^{ - \frac{|x-\overline{x}|^2 }{4 \tau} } \int_0^\infty
   e^{ - \frac{ N  t}{2 \tau} }  t^{\frac{N-2}{2}  }  u(\overline{x}, t )
  \, dt \, d \overline{x}  \\
  &=:g_N(x,\tau).
\end{align*}
Notice that
\begin{align}\label{integral-R}
   \int_0^\infty
   e^{ - \frac{ N  t}{2 \tau} }  t^{\frac{N-2}{2}  }   \, dt  &= \left( \frac{2\tau}{N} \right)^{\frac{N}{2}}   \Gamma(N/2)
\end{align}
and
$$\int_{\R^d}  e^{ - \frac{|x-\overline{x}|^2 }{4 \tau} } \, d \overline{x} =  (4 \pi \tau)^{d/2} .$$
Since $u$ is bounded, then the previous computations lead to
\begin{align*}
 | g_N(x,\tau)| & \le  \|u\|_{\infty}.
\end{align*}
We have:
\begin{align*}
  g_N(x,\tau) & - \frac{     1    }{ (4 \pi \tau)^{d/2}  }  \int_{\R^d}  e^{ - \frac{|x-\overline{x}|^2 }{4 \tau} }  u(\overline{x}, \tau ) \, d \overline{x}   =\frac{     1    }{ (4 \pi \tau)^{d/2}  }  \int_{\R^d}  e^{ - \frac{|x-\overline{x}|^2 }{4 \tau} } \cdot I(N),
\end{align*}
where
\begin{align*}
 | I(N)| & = \frac{1}{ \Gamma(N/2)}
  \left(  \frac{N}{2\tau} \right)^{N/2} \cdot   \int_0^\infty
   e^{ - \frac{ N  t}{2 \tau} }  t^{\frac{N-2}{2}  } \left( u(\overline{x}, t )  -  u(\overline{x}, \tau ) \right) \, dt.
\end{align*}
Changing variables $t=\tau \cdot a$, $dt=\tau da$, gives us
$$I(N)= \frac{1}{ \Gamma(N/2)}
  \left(  \frac{N}{2} \right)^{N/2} \cdot  \int_0^\infty
   e^{ - \frac{ N  }{2 }  a}  a^{\frac{N-2}{2}  } \left( u(\overline{x}, a \tau )  -  u(\overline{x}, \tau ) \right) \, da.
   $$
The function $a\mapsto  h(a):=e^{ - \frac{ N  }{2 }  a}  a^{\frac{N-2}{2}  } $, defined for $a \ge 0$, is increasing in $[0, \frac{N-2}{N}]$ and decreasing on $[\frac{N-2}{N},\infty)$, attaining the maximum in $a=\frac{N-2}{N}$, which is close to $1$ for large $N$. Also $\lim_{a\to \infty} h(a)=0 =h(0).$ Since $u$ is continuous at $t=\tau$, then for fixed $\epsilon>0$, there exists $\delta >0$  such that  $\left| u(\overline{x}, a\tau )  -  u(\overline{x}, \tau )\right|<\epsilon$ for $a \in (1-\delta, 1+\delta).$
This motivates us to split the integral as follows:
$$
I(N)= I_1(N) + I_2(N) + I_3(N) _= \int_{1-\delta}^{1+\delta} + \int_0^{1-\delta} + \int_{1+\delta}^{\infty}.
$$
We estimate from above each of these integrals. For $I_1(N)$ we use the continuity of $u$ and formula \eqref{integral-R}:
\begin{align*}
 | I_1(N)| & =  \frac{1}{ \Gamma(N/2)}
  \left(  \frac{N}{2} \right)^{N/2} \cdot   \int_{1-\delta}^{1+\delta}
   e^{ - \frac{ N  }{2} a }  a^{\frac{N-2}{2}  } \left| u(\overline{x}, a\tau )  -  u(\overline{x}, \tau ) \right| \, da  \le \epsilon.
  \end{align*}
  For $I_2(N)$ and $I_3(N)$ we use \eqref{Gamma:lower} that translates here as $ \Gamma(N/2)  \le  \frac{\sqrt{2\pi}}{\sqrt{N/2}} \left(\frac{N}{2}\right) ^ \frac{N}{2} e^{-\frac{N}{2}}$. Thus
 \begin{align*}
 | I_2(N)| & \le  \|u\|_\infty \frac{1}{\sqrt{2\pi}} \sqrt{N/2} \,  e^{\frac{N}{2}}\cdot \int_0^{1-\delta}   e^{ - \frac{ N  }{2 }  a}  a^{\frac{N-2}{2}  } da.
    \end{align*}
The above integral rewrites as:
$$ \int_0^{1-\delta}    e^{ - \frac{ N  }{2 }  a}  a^{\frac{N-2}{2}  } da = \left(\frac{2}{N} \right)^{\frac{N}{2}} \cdot \int_0^{\frac{N}{2}( 1-\delta)}    e^{ - x}  x^{\frac{N}{2}-1  } dx = \left(\frac{2}{N} \right)^{\frac{N}{2}} \cdot \gamma\left(\frac{N}{2},\frac{N}{2}(1-\delta) \right),$$
where $\gamma (\cdot,\cdot)$ is the incomplete Gamma function recalled in \eqref{Incomplete:gamma}.
Using  the asymptotics for large $N$ given in \eqref{Incomplete:gamma} we obtain
\begin{align*}
 | I_2(N)| &\sim \|u\|_\infty \frac{1}{\sqrt{2\pi}} \sqrt{N/2} \,  e^{\frac{N}{2}}\cdot  \left(\frac{2}{N} \right)^{\frac{N}{2}} \cdot \left( \frac{N}{2}(1-\delta) \right)^{\frac{N}{2}} \cdot e^{- \frac{N}{2}(1-\delta)}\cdot \frac{1}{\frac{N}{2} \delta} \\
 &=  \|u\|_\infty \frac{1}{\sqrt{2\pi}} \, \frac{1}{\delta \sqrt{2}} \cdot  \frac{1}{\sqrt{N}} \cdot e^{ \frac{N}{2}\delta }\cdot  \left( 1-\delta \right)^{\frac{N}{2}}.
    \end{align*}
Notice that
$$ e^{ \frac{N}{2}\delta }\cdot  \left( 1-\delta \right)^{\frac{N}{2}} = e^{ \frac{N}{2} \left( \delta + \log(1-\delta) \right) }. $$
The function $\delta \mapsto \delta + \log(1-\delta)$ defined for $\delta \in [0,1)$ is strictly decreasing (it has negative derivative) and takes $0$ value for $\delta=0$, thus it is negative for $\delta \in (0,1).$
Thus
$$I_2(N) \to 0  \quad  \text{as } N \to \infty.$$
Now we estimate $I_3(N)$. Notice that
$$ \int_{1+\delta}^\infty    e^{ - \frac{ N  }{2 }  a}  a^{\frac{N-2}{2}  } da = \left(\frac{2}{N} \right)^{\frac{N}{2}} \cdot \int_{\frac{N}{2}( 1+\delta)}^\infty    e^{ - x}  x^{\frac{N}{2}-1  } dx = \left(\frac{2}{N} \right)^{\frac{N}{2}} \cdot \Gamma\left(\frac{N}{2},\frac{N}{2}(1+\delta) \right),$$
where $\Gamma (\cdot,\cdot)$ is the incomplete Gamma function recalled in \eqref{Incomplete:gamma}.
Using the asymptotics for large $N$ given in \eqref{Incomplete:Gamma} we get
\begin{align*}
 | I_3(N)| & \le  \|u\|_\infty \frac{1}{\sqrt{2\pi}} \sqrt{N/2} \,  e^{\frac{N}{2}}\left(\frac{2}{N} \right)^{\frac{N}{2}} \cdot \Gamma\left(\frac{N}{2},\frac{N}{2}(1+\delta) \right)  \\
& \sim  \|u\|_\infty \frac{1}{\sqrt{2\pi}} \sqrt{N/2} \,  e^{\frac{N}{2}}\left(\frac{2}{N} \right)^{\frac{N}{2}} \cdot  \left(\frac{N}{2}(1+\delta) \right)^{\frac{N}{2}} e^{-\frac{N}{2}(1+\delta)} \frac{1}{\frac{N}{2}\delta} \\
&= \|u\|_\infty \frac{1}{\sqrt{2\pi} \sqrt{2} \delta} \frac{1}{\sqrt{N}} e^{-\frac{N}{2} (\delta - \log(1+\delta))}.
 \end{align*}
The function $\delta \mapsto \delta - \log(1+\delta)$ defined for $\delta \in [0,1)$ is strictly increasing (it has positive derivative) and takes $0$ value for $\delta=0$, thus it is positive for $\delta \in (0,1).$
Thus
$$I_3(N) \to 0  \quad  \text{as } N \to \infty.$$
Therefore, we conclude  that
$$ V_{N}(x,0,\tau) =g_{N}(x,\tau) \rightarrow \frac{     1    }{ (4 \pi \tau)^{d/2}  }  \int_{\R^d}  e^{ - \frac{|x-\overline{x}|^2 }{4 \tau} }  u(\overline{x}, \tau ) \, d \overline{x} =: \phi(x,\tau) \quad \text{as } \, N\to \infty .$$
Passing to the limit in \eqref{eq:VN} as $N_k \to  \infty$ we get:
\begin{equation}\label{eq:V}
  \left\{ \begin{array}{ll}
  \frac{d}{d\tau}V(x,t,\tau)=\Delta_{x} V(x,t,\tau) +  V_t(x,t,\tau) &\text{for }\tau>0,\, t >0 \text{ and } x \in \R^d,  \\[2mm]
  V(x,0,\tau)=\phi(x,\tau) &\text{for } x \in \R^d, \, \tau>0, \\[1mm]
  V(x,t,0)=u(x,t) &\text{for } x \in \R^d, \, t>0.
    \end{array}
    \right.
\end{equation}
This is a linear equation with  diffusion and transport term in the upper half-space that we consider bellow in Lemma \ref{Lemma:V}. Problem \eqref{eq:V} has a unique bounded solution given by
$$ V(x,t,\tau)=e^{-\tau (-\Delta -\partial_t)}u(x,t)=\frac{1}{(4\pi\tau)^{d/2}}  \int_{\R^d} e^{-\frac{|\overline{x}|^2}{4\tau}} u(t+\tau, x- \overline{x}) \, d \overline{x}.$$
Since the limit $ V(x,t,\tau)$ is unique,  then the whole sequence $(V_N(x,t,\tau))_N$ is convergent  to $V(x,t,\tau)$ pointwise as $N\to \infty$, for any $(x,t,\tau)$.

In summary, we have obtained that
\begin{align*}
  (-\Delta_{x,y})^s \left[ u\big(x,\frac{|y|^2}{2N}\big)  \right] &= \frac{1}{\Gamma(-s)} \int_0^ \infty \left( V_N(x,\frac{|y|^2}{2N} ,\tau) - u\big(x,\frac{|y|^2}{2N}\big)  \right) \frac{d\tau}{\tau^ {1+s}}.
    \end{align*}
Thus,  if we denote $G_N(x, \frac{|y|^2}{2N}  )= (-\Delta_{x,y})^s \left[ u\big(x,\frac{|y|^2}{2N}\big)  \right] $
 then
 $$G_N \big(x,t\big) =  \frac{1}{\Gamma(-s)} \int_0^ \infty \left( V_N (x,t ,\tau) - u\big(x,t\big)  \right) \frac{d\tau}{\tau^ {1+s}}.$$
By letting $N\to \infty$ and in view of definition \ref{defn:FractBackwardHeat}, we obtain
$$G_{N} (x,t) \rightarrow \frac{1}{\Gamma(-s)} \int_0^ \infty \left( V (x,t ,\tau) - u\big(x,t\big)  \right) \frac{d\tau}{\tau^ {1+s}}
=  ( -\Delta_x-\partial_t)^{s} [u(x,t)]. $$

\end{proof}

\begin{lemma}\label{Lemma:V}
 The solution of \eqref{eq2:V-bis}
is unique within the class of bounded solutions and it is given by
\begin{equation*}
   V(x,t,\tau)=e^{-\tau (-\Delta -\partial_t)}u(x,t)=\frac{1}{(4\pi\tau)^{d/2}}  \int_{\R^d} e^{-\frac{|\overline{x}|^2}{4\tau}} u(t+\tau, x- \overline{x}) \, d \overline{x}.
   \end{equation*}
\end{lemma}
\begin{proof}
  Since the equation is linear, then it is sufficient to prove that $0$ is the unique solution to
    \begin{equation*} 
  \left\{ \begin{array}{ll}
 \displaystyle \frac{d}{d\tau}\mathcal{V}(x,t,\tau)=\Delta_{x} \mathcal{V}(x,t,\tau) +  \mathcal{V}_t(x,t,\tau) &\text{for }\tau>0,\, t >0 \text{ and } x \in \R^d,  \\[3mm]
  \mathcal{V}(x,0,\tau)= 0  &\text{for } x \in \R^d, \, \tau>0, \\[1mm]
  \mathcal{V}(x,t,0)=0 &\text{for } x \in \R^d, \, t>0.
    \end{array}
    \right.
\end{equation*}
%
We use the method of characteristics (see for instance \cite[Ch.1]{Fritz-John}). Let $(t,\tau) \in \{ \tau +t =C \} $  for a fixed $C$.  Let $h(x,\tau)= \mathcal{V}(x,C-\tau, \tau).$ Then $h$ solves the PDE:
 \begin{equation*}
  \left\{ \begin{array}{ll}
\displaystyle \frac{d}{d\tau}h(x,\tau)  = \Delta_x h(x,\tau),  &\text{for }\tau>0 \text{ and } x \in \R^d,  \\[3mm]
 h(x,0)= 0  &\text{for } x \in \R^d.
    \end{array}
    \right.
\end{equation*}
Since the solution of the heat equation is unique within the class of functions that satisfy at most some exponential upper bound  (see, for instance, \cite{Evans}), and since our case is simpler as we work with bounded solutions, it follows that $h=0$. Consequently, $\mathcal{V}=0$ along the lines $\tau +t =C$, and we conclude that $V=0$.
\end{proof}

\begin{remark}
  The transformation given in Proposition \ref{prop:transf} works in fact for bounded  functions $u(x,t)$ with less regularity such that $(-\Delta_{x,y})^s \left[ u\big(x,\frac{|y|^2}{2N}\big)  \right]$  and  $ (-\Delta_x-\partial_t)^{s} [u(x,t)] $ are well defined and such that the associated heat equation used above has classical solutions.
\end{remark}

\section{Proof of Theorems \ref{main:thm} and \ref{Thm:CarlemanLp}}

\begin{proof}(of Theorem \ref{main:thm}) Let $u: \R^d \times [0,\infty) \to \R$ as in the hypothesis.  Let $N\ge 1$. We define the function  $f:\R^d \times \RN \to \R$  by
$$f(x,y)=u\big(x,\frac{r^2}{2N}\big) ,\quad \text{where } r=|y|,$$
and thus $f$ is radial with respect to the $y$ coordinate. As proved in Proposition \ref{prop:transf},   we have
\begin{equation*}
(-\Delta_{x,y})^s f(x,y) =   G_N\big(x,  \frac{|y|^2}{2N}\big),
\end{equation*}
for some $G_N:\R^d \times [0,\infty)\to \R$.   According to \cite[Thm16]{EilertsenJFA2001}, that we recalled in Lemma \ref{Lemma:Eilertsen} and we adapt it here to $N+d$ dimensions, the following weighted estimate holds:
$$
 \int_{\R^{N+d}} |f(x,y)|^2 |(x,y)|^{2\eta - N-d} dx dy \le C_{N,d} \int_{\R^{N+d}} \left| (-\Delta_{x,y})^{s}  f(x,y)\right|^2  |(x,y)|^{ 4s +2\eta - N-d} dx\, dy
 $$
for $s \ge 0$, $\alpha=0$ and  $\eta, N+d-2s-\eta\not\in -\mathbb{N}$, where the best constant $C_{N,d}$ is given in formula \eqref{CNd}.
Equivalently, we obtain:
\begin{align*}
 \int_{\R^{N+d}} |f(x,y)|^2 & (|x|^2+|y|^2)^{\frac{{2\eta - N-d}}{2}} dx dy  \\
 & \le C_{N,d} \int_{\R^{N+d}} \left| (-\Delta_{x,y})^s  f(x,y)\right|^2  \cdot(|x|^2+|y|^2)^{\frac{ 4s +2\eta - N-d}{2}} dx \, dy.
  \end{align*}
    By passing to polar coordinates with $r= |y| $ in both left and right hand side integrals, and using notation \eqref{FractLap:f}, we rewrite as follows:
\begin{align*}
&\int_{\R^{d}}\int_0^\infty \left| u(x,\frac{r^2}{2N} ) \right|^2 \cdot \left(|x|^2+r^2\right)^{\frac{{2\eta - N-d}}{2}}\cdot r^{N-1}\, dr \, dx \\
&\le
C_{N,d} \int_{\R^{d}}\int_0^\infty  \left| G_N(x,\frac{r^2}{2N}  )  \right|^2  \cdot(|x|^2+r^2)^{\frac{ 4s +2\eta - N-d}{2}} r^{N-1} dr dx.
\end{align*}
Equivalently, we obtain:
\begin{align*}
&\int_{\R^{d}}\int_0^\infty \left| u(x,\frac{r^2}{2N} ) \right|^2 \cdot \left( \frac{|x|^2}{r^2}+1 \right)^{\frac{{2\eta - N-d}}{2}} \cdot  r^{ 2\eta-d -1} \,dr \, dx \\
&\le C_{N,d} \int_{\R^{d}}\int_0^\infty  \left| G_N(x,\frac{r^2}{2N}  ) \big] \right|^2  \cdot\left(\frac{|x|^2}{r^2}+1 \right)^{\frac{ 4s +2\eta - N-d}{2}} r^{4s +2\eta -d-1} dr \, dx.
\end{align*}
In the left hand side integral we perform the change of variables $r^2/2N =t$. Then $r \, dr =N  \, dt$. Thus,
 after simplifying constants, we obtain:
\begin{align}\label{ineqN}
&\int_{\R^{d}}\int_0^\infty |u(x,t )|^2 \cdot \left( \frac{|x|^2}{2 N t}+1 \right)^{\frac{{2\eta - N-d}}{2}}    t^{\frac{  2\eta-d -2 }{2}}  dt dx \\
&\le \widetilde{C}_{N,d}  \int_{\R^{d}}\int_0^\infty  \left| G_N(x,t  )  \right|^2  \cdot \left(\frac{|x|^2}{2N t}+1 \right)^{\frac{ 4s +2\eta - N-d}{2}}  t^{\frac{ 4s +2\eta -d-2}{2}} dt \, dx. \nonumber
\end{align}
where $\widetilde{C}_{N,d}= (2N )^{2s} C_{N,d}$.

\noindent\emph{Step 3. Passing to the limit $N\to \infty$.} Then
$$
 \left( \frac{|x|^2}{2 N t}+1  \right)^{\frac{{2\eta - N-d}}{2}}  \to e ^{-\frac{|x|^2}{4t}} \quad \text{as }  N \to \infty.
 $$
We are now interested in the limit as $N \to \infty$ of the family  of functions $\{G_N(\cdot,\cdot)\}_{N\in \mathbb{N}^*}$.
From Proposition \ref{prop:transf}, we know that
$G_N(x, t) \rightarrow (-\partial_t-\Delta_x)^s [u(x,t)]$ pointwise for every $(x,t)$.
 Passing to the limit inside the integral can be justified via the Dominated Convergence Theorem since $u \in \mathcal{S}$.

Finally, we pass to the limit with the above constants. By formula \eqref{CNd} adapted to $N+d$ dimensions, we obtain:
\begin{align*}
   \widetilde{C}_{N,d} & = (2N )^{2s}  C_{N,d} = \left(\frac{N}{2}\right) ^{2s}    \max_{j=j_0, j_0+1} \left(  \frac{ \Gamma\big( \frac{\eta+j}{2} \big)   \Gamma\big( \frac{N+d-2s-\eta+j}{2} \big)   }
{ \Gamma\big( \frac{2s+\eta+j}{2} \big)   \Gamma\big( \frac{N+d-\eta+j}{2} \big)   } \right)^2,
\end{align*}
where $j_0=j_0(N)$ be the smallest non-negative integer satisfying
\begin{equation}\label{eq:j0}
(N+d - 2\eta )( N+d-4s-2\eta ) \le (N+d + 2j_0)^2.
\end{equation}
In order to study this constant, we show  first that $j_0$ defined by previous formula is bounded by a constant depending on $\eta$ and $\beta$. Let us write inequality \eqref{eq:j0} in the following equivalent formulation:
\begin{equation}\label{eq:j0-2}
 2 \eta(2s+\eta) \le (N+d) ( j_0 + s + \eta) + 2 j_0^2.
\end{equation}
 Indeed, observe that, if we take $j_1$ to be the smallest non-negative integer satisfying
$$0\le j_1 + s + \eta  \quad  \text{and}  \quad    \eta(2s+\eta) \le   j_1^2 ,  $$
then $j_1$ satisfies \eqref{eq:j0-2}, and thus $j_0 \in \{ 0,1,\dots, j_1 \}$. Notice that $j_1$ depends only on $\eta$ and $s$, and it does not depend on $N$ and $d$.
Thus, for all $N\ge 1$, we have $1\le j_0\le j_1(\eta,s)$.
Moreover, if $s+\eta \ge 0$ then $j_1 =j_0=0$.

Using the inequality \eqref{ineq:Gamma} we obtain that
$$ \left(\frac{N}{2}\right) ^{s} \cdot  \frac{    \Gamma\big( \frac{N+d-2s-\eta+j}{2} \big)   }
{   \Gamma\big( \frac{N+d-\eta+j}{2} \big)   }
\le\left(\frac{N}{2}\right) ^{s} \cdot  \left(   \frac{N+d-\eta+j-2}{2} \right)^{-s}   \to  1 \quad \text{as }N \to \infty.
$$
In fact, the quotient in the left hand side tends asymptotically to $1$ using \eqref{asymp:GammaQuotient}, thus it will not extinguish when passing to the limit. This implies that
$$ \widetilde{C}_{N,d}  \le    \widetilde{C}(\eta,s), $$
where
$$
\widetilde{C}(\eta,s) = \max_{j \in \{ 0,1,\dots, j_1+1 \}} \left(  \frac{ \Gamma\big( \frac{\eta+j}{2} \big)   }
{ \Gamma\big( \frac{2s+\eta+j}{2} \big)} \right)^2,$$

Therefore, by passing to the limit in \eqref{ineqN}, we obtain
\begin{align*}
&\int_{\R^{d}}\int_0^\infty |u(x,t )|^2 \cdot e ^{-\frac{|x|^2}{4t}}     t^{\frac{  2\eta-d -2 }{2}}  dt\,  dx \\
&\le  \widetilde{C}(\eta,s) \int_{\R^{d}}\int_0^\infty  \left| ( -\Delta_x-\partial_t)^{s} [u(x,t)]  \right|^2  \cdot e ^{-\frac{|x|^2}{4t}}   t^{2s + \frac{2\eta -d-2}{2}} dt \, dx.
\end{align*}

If $s=1$ and $\frac{  2\eta-d -2 }{2} <0$ then we recover the result for the Laplacian.

\end{proof}

\begin{proof}(of Theorem \ref{Thm:CarlemanLp})
We use Lemma \ref{Lemma:deNitti} adapted to $\R^{N+d}$ variables and  denoting $\theta=N+d-2\eta-2s$ in order to unify notation with the previous theorem.    We follow the same steps as in proof of Theorem \ref{main:thm} and we arrive at:
  \begin{align*}
  \left(\frac{b_{N,d}}{p}\right)^p
&\int_{\R^{d}}\int_0^\infty |u(x,t )|^p \cdot \left( \frac{|x|^2}{2 N t}+1 \right)^{\frac{{2\eta - N-d}}{2}}    t^{\frac{  2\eta-d -2 }{2}}  dt \, dx \\
&\le (2N)^{sp} \int_{\R^{d}}\int_0^\infty  \left| G_N(x,t  )  \right|^p  \cdot \left(\frac{|x|^2}{2N t}+1 \right)^{\frac{ 2\eta +2sp - N-d}{2}}  t^{\frac{ 2\eta -d+2sp-2}{2}} dt\, dx.
\end{align*}
  We estimate the constant for large $N$:
  \begin{align*}
    (2N)^{-sp}   \left(\frac{b_{N,d}}{p}\right)^p  & =  (2N)^{-sp} \,
    2^{2sp}  \cdot \left[ \frac{1}{p}
  \frac{
  \Gamma\left(  \frac{ - 2\eta-2s}{2} \right) \cdot  \Gamma\left( \frac{N+d-2\eta}{2}\right)
  }
  {
   \Gamma \left(\frac{2\eta-d}{2} \right) \cdot  \Gamma\left(  \frac{N+d-2\eta-2s}{2}\right)
   } \right]^p  \\
    &= \left[\frac{1}{p}
      \frac{
  \Gamma\left(  \frac{ - 2\eta-2s}{2} \right)
  }
  {
   \Gamma \left(\frac{2\eta-d}{2} \right)
   } \right]^p =:\left(\frac{\tilde{b}_{d,\eta,s}}{p}\right)^p.
      \end{align*}
    By passing to the limit $N\to \infty$  we obtain:
  \begin{align*}
&  \left(\frac{\tilde{b}_{d,\eta,s}}{p}\right)^p \int_{\R^{d}}\int_0^\infty |u(x,t )|^p   \cdot e ^{-\frac{|x|^2}{4t}}     t^{\frac{  2\eta-d -2 }{2}}  dt\, dx \\
&\le   \int_{\R^{d}}\int_0^\infty  \left| ( -\Delta_x-\partial_t)^{s} [u(x,t)]  \right|^p  \cdot e ^{-\frac{|x|^2}{4t}}   t^{ sp + \frac{2\eta -d-2}{2}} dt \, dx.
\end{align*}
\end{proof}

\section{Comments and Open Problems}

\begin{enumerate}

  \item It would be very interesting to prove the strong unique continuation using Theorem \ref{main:thm} as in local case done in \cite{EscauriazaDuke2000}. The main problem remains the  "localisation"  of functions when applying $(-\Delta_x-\partial_t)[\rho(x,t) u(x,t)]$ for some cut-off function $\rho$.  A fine quantitative estimate for $ (-\Delta_x-\partial_t)[\rho(x,t)]$ could be tried.

        \item In \cite[Lemma 2]{FernandezCPDEs2003}, Fern\'andez proves that if $u$ solves $(-\Delta -\partial_t) u =V u $, for $V$ sufficiently good,  and if $u(x,0)$  decays of infinite order with respect to the $x$ variable:
        \begin{equation}\label{decay-u-2}
    |u(x,0)| \le C_k |x|^k, \quad \text{ for all } k \in \mathbb{N}, \, (x,t)\in B_2 \times [0,2],
  \end{equation}
      then $u(x,0)=0$ in $B_2$.

   Theorem \ref{main:thm}  can be applied to functions with good decay in the time variable. It would be useful to prove estimate \eqref{Carleman} for $t+a$  instead of $t$ with some $a>0$, as done in \cite[Lemma 2]{FernandezCPDEs2003} for $s=1$, and, hopefully, to be able to use this estimate to prove the strong unique continuation for solutions to $(-\Delta -\partial_t)^s u =V u $ when assuming only a condition like \eqref{decay-u-2}.
    Using our strategy, we were not able to complete the proof of the  estimate \eqref{Carleman} with $t+a$.

  \end{enumerate}

\section{Appendix}

$\bullet$ The Gamma function satisfies the lower bound (we refer to \cite[formulla (5.6.1)]{OlMax})
\begin{equation}\label{Gamma:lower}
  1 < (2\pi)^{-\frac{1}{2}} \, x^{\frac{1}{2}-x} \, e^x \, \Gamma(x),\quad \quad \text{for all }x>0.
\end{equation}
Also, the quotient of Gamma functions satisfies (see \cite[formulla (5.6.4)]{OlMax})
\begin{equation}\label{ineq:Gamma}
  \frac{\Gamma(x+s)}{ \Gamma(x+1)} < x^{s-1},  \quad \text{for all } s\in (0,1),\, x>0.
\end{equation}

$\bullet$ The asymptotics for large argument of quotient of Gamma functions is given by (see \cite[Ch.~5, formula (5.11.12)]{OlMax}):
\begin{equation}\label{asymp:GammaQuotient}
\frac{\Gamma(z+a)}{\Gamma(z+b)} \sim z^{a-b},\quad \text{as } z \to +\infty.
\end{equation}

$\bullet$ The incomplete gamma function (see \cite[formula (8.11.6)]{OlMax}) for $z=\lambda a$ with $\lambda \in (0,1)$ fixed, as $a\to \infty$ behaves as
\begin{equation}\label{Incomplete:gamma}
  \gamma(a,z)= \int_{0}^{z} x^{a-1}e^{-x} dx  \sim   z^a e^{-z} \cdot \frac{1}{a-z}.
\end{equation}

$\bullet$ The incomplete gamma function (see \cite[formula (8.11.7)]{OlMax}) for $z=\lambda a$ with $\lambda > 1 $ fixed, as $a\to \infty$ behaves as
\begin{equation}\label{Incomplete:Gamma}\Gamma(a,z)= \int_{z}^\infty x^{a-1}e^{-x} dx  \sim   z^a e^{-z} \cdot \frac{1}{z-a}.
\end{equation}


$\bullet$ The modified Bessel function of first kind, for $\nu>-\frac{1}{2}$ can be represented as (\cite[formula (10.32.2)]{OlMax})
\begin{equation}\label{BesselI}
  I_\nu (z)= \frac{(\frac{1}{2}z) ^\nu}{ \sqrt{\pi} \, \Gamma(\nu + \frac{1}{2})} \int_{-1}^{1} (1-t^2)^{\nu - \frac{1}{2}} e^{\pm zt} dt.
\end{equation}

$\bullet$ Area of the $(N-1)$- sphere  $S^{N-1} \subset \mathbb{R}^N$ is
\begin{equation}\label{area:sphere}
  |S^{N-1}|= \frac{2 \pi^{N/2}}{\Gamma\left(\frac{N}{2}\right)}.
\end{equation}

$\bullet$ Let $a\in S^{N-1}$, after a change of variables of type $R \tilde{\sigma}=\sigma$ where $R$ is a rotation on $S^{N-1}$ that takes $a$ into $e_1$:
\begin{align}\label{Integral:Rotation}
\int_{S^{N-1}} g (a\cdot\sigma) d\sigma &= \int_{S^{N-1}} g ( e_1\cdot \tilde{\sigma} ) d\sigma  =   \int_{-1}^1 g (t) \int_{|\xi|=\sqrt{1-t^2}} 1 \, d\xi  \,dt\\  \nonumber
&=
| S^{N-2} | \int_{-1}^1 g (t)(1-t^2)^{  \frac{N-2}{2} }\, d\xi  \,dt .
\end{align}
We used cartesian coordinates for the integral $ (t,\xi) \in \R \times \R^{N-1}$ with $ t^2 + |\xi|^2 =1$. The set $\{ |\xi|=\sqrt{1-t^2} \} $ is the sphere of radius $\sqrt{1-t^2}$ in $\R^{N-1}$.

\bigskip

\bigskip

\noindent{\textbf{Acknowledgments}.} This work is motivated by many interesting discussions with Prof.  Luis Escauriaza in 2017 on the local case $s=1$, who pointed to the author references \cite{DaveyARMA2018} and \cite{EilertsenJFA2001}.

The author is grateful to Aingeru Fernández Bertolin and  Luz Roncal for useful comments on the manuscript.

\bibliographystyle{siam}

\end{document}